\documentclass[12pt]{amsart}
\usepackage[all]{xy}

\usepackage{fullpage}
\usepackage{setspace}
\usepackage{amsmath,amsthm,amssymb, paralist, xspace, graphicx, url, amscd, euscript, mathrsfs,epic,eepic,verbatim,stmaryrd}
\usepackage{caption}

\usepackage[usenames,dvipsnames]{color}

\usepackage[colorlinks=true]{hyperref}

\DeclareMathOperator{\lcm}{lcm}

\numberwithin{equation}{section}

1

\numberwithin{subsection}{section}

\allowdisplaybreaks[1]

\newtheorem*{namedtheorem}{\theoremname}
\newcommand{\theoremname}{testing}

\newtheorem{theorem}{Theorem}[section]
\newtheorem{proposition}[theorem]{Proposition}
\newtheorem{proposition-definition}[theorem]
{Proposition-Definition}
\newtheorem{corollary}[theorem]{Corollary}
\newtheorem{lemma}[theorem]{Lemma}
\newtheorem{conjecture}[theorem]{Conjecture}

\theoremstyle{definition}
\newtheorem{definition}[theorem]{Definition}
\newtheorem{remark}[theorem]{Remark}

\renewcommand{\thesubsubsection}{\ifnum\value{subsection}=0
	\arabic{section}.\arabic{subsubsection}%
\else
	\arabic{section}.\arabic{subsection}.\arabic{subsubsection}%
\fi}

\makeatletter
\@addtoreset{subsubsection}{section}

\let\c@equation\c@subsubsection

\let\subsection@old\subsection
\def\subsection#1{\ifnum\value{subsubsection}>0 \ifnum\value{subsection}=0
	\setcounter{subsection}{\value{subsubsection}}%
\fi \fi
\subsection@old{#1}}
\makeatother

\theoremstyle{remark}

\newcommand{\tX}{\widetilde{X}}

\newcommand{\sD}{\mathscr{D}}

\newcommand\cA{\mathcal{A}}

\newcommand\cM{\mathcal{M}}

\newcommand\cO{\mathcal{O}}

\newcommand\cR{\mathcal{R}}

\newcommand{\oX}{{\overline{X}}}

\newcommand{\ocA}{{\overline{\mathcal{A}}}}
\newcommand{\tcA}{{\widetilde{\mathcal{A}}}}
\newcommand{\otcA}{{\overline{\widetilde{\mathcal{A}}}}}
\newcommand{\ocM}{{\overline{\mathcal{M}}}}

\renewcommand\AA{\mathbb{A}}

\newcommand\CC{\mathbb{C}}

\newcommand\GG{\mathbb{G}}

\newcommand\NN{\mathbb{N}}

\newcommand\QQ{\mathbb{Q}}
\newcommand\RR{\mathbb{R}}

\newcommand\ZZ{\mathbb{Z}}

\newcommand\Spec{\operatorname{Spec}}
\newcommand\Spf{\operatorname{Spf}}

\begin{document}

\title[Level structures and Lang's conjecture]{Level structures on abelian varieties, \\
Kodaira dimensions, and Lang's conjecture}

\author[D. Abramovich]{Dan Abramovich}
\address{Department of Mathematics, Box 1917, Brown University, Providence, RI, 02912, U.S.A}
\email{abrmovic@math.brown.edu}

\author[A. V\'arilly-Alvarado]{Anthony V\'arilly-Alvarado}
\address{Department of Mathematics MS 136, Rice University, 6100 S.\ Main St., Houston, TX 77005, USA}
\email{av15@rice.edu}


\thanks{Research by D. A. partially supported by NSF grant DMS-1500525. Research by A. V.-A. partially supported by NSF CAREER grant DMS-1352291. This paper began as a lunch conversation during the workshop ``Explicit methods for modularity of K3 surfaces and other higher weight motives'', held at ICERM in October, 2015. We thank the organizers of the workshop and the staff at ICERM for creating the conditions that sparked this project. We also thank {\sc Y. Brunebarbe, F. Campana, G. Farkas, B. Hassett, K. Hulek, J. Koll\'ar, R. Lazarsfeld,  B. Mazur, M. Popa, C. Schnell, J. Silverman, E. Ullmo, D. Ulmer, J. Voight, D. Zureick-Brown}, and the anonymous referees, who gave insightful comments and suggested numerous references.}

\date{\today}

\begin{abstract}
Assuming Lang's conjecture, we prove that for a prime $p$, number field $K$, and positive integer $g$, there is an integer $r$ such that no principally polarized abelian variety $A/K$ has full level-$p^r$ structure. To this end, we use a result of Zuo to prove that for each closed subvariety $X$ in the moduli space $\cA_g$ of principally polarized abelian varieties of dimension $g$, there exists a level $m_X$ such that the irreducible components of the preimage of $X$ in $\cA_g^{[m]}$ are of general type for $m > m_X$. 
\end{abstract}

\maketitle

\setcounter{tocdepth}{1}
\tableofcontents


\section{Introduction}


\subsection{Main result: arithmetic}

	{\sc Yuri Manin} proved in \cite{Manin} that, given a number field $K$ and a prime $p$,  the order  of $p$-primary torsion points across \emph{all} elliptic curves over $K$ is bounded. Our main arithmetic result is an analogous statement for higher dimensional abelian varieties, conditional on {\sc Lang}'s conjecture (\cite[Conjecture 5.7]{Lang}, see Conjecture~\ref{conj:Lang} below). Instead of $p$-primary torsion, we treat the more tractable case of full level structures: a {\em full level-$m$ structure} on an abelian variety $A$ of dimension $g$ is an isomorphism of group schemes on the $m$-torsion subgroup
	\[
		A[m] \,\xrightarrow{\ \sim\ }\, (\ZZ/m\ZZ)^g \times (\mu_{m})^g.
	\]
We do not require this isomorphism to be compatible with the Weil pairing.

	\begin{theorem}[Uniform power bound]\label{Th:main}
		Assume that {\sc Lang}'s conjecture holds. Fix an integer $g$, a prime number $p$, and a number field $K$. Then there is an integer $r$ such that no principally polarized abelian variety $A/K$ of dimension $g$ has full level-$p^r$ structure.
	\end{theorem}

See  \S\ref{Sec:arithmetic} for known results and variants of Theorem~\ref{Th:main}. The main ingredient in our proof is a powerful result of {\sc Zuo}~\cite{Zuo}\footnote{In an earlier version of this article, we used instead a recent result of {\sc Popa} and {\sc Schnell} \cite{Popa-Schnell} to get our main argument off the ground. Their result applies to arbitrary families of polarized varieties, not necessarily of Torelli type.}. The complex function field analogue of our result is shown unconditionally by {\sc Hwang} and {\sc To} in \cite[Theorem 1.3]{Hwang-To-uniform}. See also {\sc Rousseau} \cite{Rousseau} and {\sc Bakker-Tsimerman} \cite[Theorem A]{Bakker-Tsimerman-torsion}. 

	Theorem~\ref{Th:main} is a byproduct of our ongoing pursuit of analogous results for K3 surfaces. That investigation follows on unconditional results of {\sc V\'arilly-Alvarado}, with {\sc McKinnie, Sawon} and {\sc Tanimoto} in \cite{MSTV} and with {\sc Viray} in \cite{VAVBrauer}. However, Theorem~\ref{Th:main} is certainly closer in spirit to numerous unconditional results, of both geometric and arithmetic nature, of {\sc A. Cadoret} and  {\sc A. Tamagawa}, as well as {\sc Ellenberg, Hall}, and {\sc Kowalski}; see e.g. \cite{Cadoret-Tamagawa-uniform,Cadoret-Tamagawa-torsion}, \cite[Theorem~7]{EHK}, and especially {\sc Cadoret}'s conditional result \cite{Cadoret}.
	
	\begin{remark}
		{\sc Zarhin}'s trick~\cite{Zarhin} allows one to remove the hypothesis that the abelian varieties in Theorem~\ref{Th:main} be principally polarized. We leave the necessary modifications to the interested reader.
	\end{remark}

\begin{remark}
It can be shown that {\sc Vojta}'s conjecture~\cite[Conjecture~2.3]{VojtaABC} implies that for large $m$ no principally polarized abelian variety $A/K$ of dimension $g$ has full level-$m$ structure. We address this implication in a separate note \cite{AV-Campana-Vojta}.
	\end{remark}


\subsection{Main results: geometry} 

	We apply {\sc Lang}'s conjecture to certain varieties of general type lying within moduli spaces, in the style of landmark results \cite{CHM} of {\sc Caporaso, Harris} and {\sc Mazur}. We work with moduli spaces of abelian varieties with level structure.

	Denote by $\cA_g^{[m]}$ the moduli space of principally polarized abelian varieties of dimension $g$ with full level-$m$ structure. When $m=1$ write  $\cA_g:= \cA_g^{[1]}$, and let $\pi_m\colon \cA_g^{[m]} \to \cA_g$ be the natural morphism that ``forgets the level structure". The morphism $\pi_m$ is finite.

	The geometric result on which Theorem \ref{Th:main} relies is:
	
	\begin{theorem}[Eventual hyperbolicity]\label{Th:eventual-hyperbolicity}
		Let $X\subset \cA_g$ be a locally closed subvariety. There is an integer $m_X$ so that for all $m>m_X$, every irreducible component of  $\pi_m^{-1} X \subset \cA_g^{[m]}$ is of general type.
	\end{theorem}
	
	\begin{remark}
		The special case when $X = \cA_g$ is known. See~\cite[Theorem~1.1]{Hulek}, where precise values of $m_X$ are given for each $g$ in this case; these values are sharp, with the possible exception of $g = 6$, because the {Kodaira} dimension of $\cA_6$ is still unknown.
	\end{remark}
	
	\begin{remark} {\sc Nadel} \cite{Nadel} proved that the {\sc Baily-Borel} compactification of $\cA_g^{[m]}$ is {\sc Brody}-hyperbolic, and {\sc Noguchi} \cite {Noguchi} refined that result. {\sc Ullmo-Yafaev} \cite{Ullmo-Yafaev} note that another conjecture of {\sc Lang} \cite[Conjecture 2.3]{Lang},  to the effect that   a {\sc Brody}-hyperbolic projective variety over a number field $K$ has finitely many rational points, implies that $\cA_g^{[m]}(K)$ is finite for large $m$. {\sc Lang} also conjectured \cite[Conjecture 5.6]{Lang}, that a  projective variety is {\sc Brody}-hyperbolic if and only if every subvariety is of general type, but this remains  open in both directions.
	\end{remark} 

	\begin{remark} \label{Rem:Brunebarbe-new} While this paper was under review, {\sc Brunebarbe} posted \cite{Brunebarbe-new}, in particular providing in \cite[Theorem 1.6]{Brunebarbe-new} the stronger Theorem \ref{Th:strong-eventual-hyperbolicity} stated below. 
\end{remark}
	Since $\cM_g \subset \cA_g$ is locally closed, Theorem~\ref{Th:eventual-hyperbolicity} immediately implies:

	\begin{corollary}
		\label{cor:Mg}
		For any positive integer $g$ there is an integer $m_g$ such that for $m > m_g$ the moduli space $\cM_g^{[m]}$ of curves of genus $g$ with {\em abelian} full level-$m$ structure is of general type.
	\end{corollary}
	
	{\sc Brylinski} has shown a result analogous to Corollary~\ref{cor:Mg} for dihedral level structures~\cite{Brylinski}. {\sc Gavril Farkas} pointed out to us that a stronger result follows from the explicit computations of~\cite{ChEFS}: the moduli space $\cR_{g,m}$ of curves with a single $m$-torsion point on their Jacobians is of general type for large $m$. For instance, their Theorem 0.2 shows that $\cR_{g,3}$ is of general type once $g\geq 12$, and similar explicit bounds can be obtained in lower genera.


\subsection{Discussion: geometry}


	\subsubsection*{Moduli stacks}

	The proof of Theorem \ref{Th:eventual-hyperbolicity} follows  {\sc Mumford}'s ideas in \cite{Mumford}. To avoid convoluted allusions to ``orbifold structures", we systematically deploy the language of algebraic stacks.

	Denote by $\tcA_g^{[m]}$ the moduli {\em stack} of principally polarized abelian varieties of dimension $g$ with full level-$m$ structure; see \cite[I.4.11, IV.6.2(c)]{Faltings-Chai}. Its coarse moduli space in the sense of  \cite{Keel-Mori} is $\cA_g^{[m]}$; it is a fine moduli space provided $m \geq 3$, so $\cA_g^{[m]} = \tcA_g^{[m]}$ in this case. For $m=1$ we simply write $\tcA_g:= \tcA_g^{[1]}$.
 

	\subsubsection*{Logarithmic hyperbolicity}
	
	The proof of eventual hyperbolicity relies on the following:

	\begin{theorem}[Logarithmic hyperbolicity]\label{Th:log-hyp}
		Let $X \subset \tcA_g$ be a closed substack, $X'\to X$ a resolution of singularities, $X' \subset \overline X'$ a smooth compactification with $D = \overline X' \smallsetminus X'  $ a normal crossings divisor.  Then $K_{\overline X'}+D$ is big.
	\end{theorem}

	Theorem \ref{Th:log-hyp} says that, after some technical choices of resolutions and compactifications, a substack of the moduli stack $\tcA_g$ is of logarithmic general type: its canonical divisor class might not be big, but adding the normal crossings boundary divisor returns a big class.

	In the case where $X$ is a scheme, logarithmic hyperbolicity follows directly from a general result of {\sc Zuo}~\cite[Theorem~0.1(ii)]{Zuo}.  We apply {\sc Zuo}'s result to subschemes of $\tcA_g^{[m]}$ with any $m\geq 3$ in Proposition \ref{Prop:hyp-m}, and we use a descent argument to prove Theorem \ref{Th:log-hyp}.

	Big divisors can be perturbed slightly and remain big, yielding:

	\begin{corollary}\label{Cor:epsilon}
		There is an $\epsilon>0$ such that  $K_{\overline X'}+ (1-\epsilon)D$ is big.
	\end{corollary}

An immediate sample  result in the spirit of Corollary \ref{cor:Mg} is the following:

	\begin{corollary}
		\label{cor:KMg}
		Let $\Delta_0\subset \ocM_g$ be the boundary divisor of irreducible nodal curves on the Deligne--Mumford stack of stable curves. Then $K_{\ocM_g}+\Delta_0$ is big. 
	\end{corollary}

{\sc Gavril Farkas} also pointed out that Corollary~\ref{cor:KMg} is also an easy consequence of the well-known relationships between divisors on $\ocM_g$: the divisor $K_{\ocM_g}+\Delta_0$ can be easily written as a positive combination of the hodge class $\lambda$, which is big and nef, and the boundary divisors $\Delta_i$ and the Brill-Noether divisor $BN_g$, which are effective.

	\subsubsection*{Toroidal compactifications}

	To get from logarithmic hyperbolicity in the form of~Corollary~\ref{Cor:epsilon} to eventual hyperbolicity, we must bridge the gap between the perturbed logarithmic canonical class $K + D-\epsilon D$ and an actual canonical class $K + D - D = K$. {\sc Mumford}'s idea in \cite[Proposition~4.4]{Mumford} is that, by adding level structure to the ambient stack $\tcA_g$, we can force enough ramification over $D$ in the moduli stack with level structure to spin straw ($-\epsilon D$) into gold ($-D$). 

	To make this idea precise, we must work with toroidal compactifications.  Following \cite{AMRT,Faltings-Chai} one can choose smooth toroidal compactifications $\tcA_g^{[m]} \subset \otcA_g^{[m]}$, with coarse moduli spaces $\cA_g^{[m]} \subset \ocA_g^{[m]}$, compatibly in towers: the normalization  of $\otcA^{[m]}$ in $\tcA_g^{[md]}$ is a toroidal compactification  $\otcA_g^{[md]}$  of $\tcA_g^{[md]}$ which turns out to be smooth, and similarly for coarse moduli spaces; see \cite[p.~128]{Faltings-Chai}. Indeed, a toroidal compactification of $\tcA_g^{[m]}$ is determined by a $GL_n(\ZZ)(m)$-equivariant  polyhedral decomposition of the cone of positive-definite matrices, where 
	\[
		GL_n(\ZZ)(m) = \ker\left( GL_n(\ZZ) \to GL_n(\ZZ/m\ZZ) \right).
	\]
	The compactification is smooth  if the cones satisfy the condition of \cite[Theorem 4, p. 14]{KKMS}, which can always be achieved by a further equivariant subdivision, see \cite[Theorem 11*, p. 94]{KKMS}.
A $GL_n(\ZZ)$-equivariant subdivision is automatically $GL_n(\ZZ)(m)$-equivariant, and the construction of the toroidal compactification gives a lifting of $\tcA_g^{[m]} \to \tcA_g$ to the compactification  $\otcA_g^{[m]} \to \otcA_g$. 


	\subsubsection*{Ramification}

	The bridge between logarithmic hyperbolicity to eventual hyperbolicity is the following:

	\begin{theorem}[Straw into gold] \label{Prop:GRH}
		Let $X$, $X'$, $\overline X'$, and $D$ be as in Theorem~\ref{Th:log-hyp}. Let $\oX'_m \to\oX'\times_{\otcA_g} \otcA_g^{[m]}$ be a resolution of singularities  with projection $\pi_m^X\colon \oX'_m\to \overline{X}'$. Assume that the rational map $\oX'_m \to \otcA_g$ is a morphism. Then for any $\epsilon > 0$ there is an integer $m_X$ such that for every $m>m_X$ we have $$({\pi^X_m})^* (K_{\overline X'} + (1-\epsilon)D)\ \  \leq\ \  K_{\oX'_m}.$$ 
	\end{theorem}

	Note that Theorem~\ref{Th:eventual-hyperbolicity} is an immediate consequence of Corollary~\ref{Cor:epsilon} and Theorem~\ref{Prop:GRH}. 
	We prove Theorem ~\ref{Prop:GRH} using the following general principle, see Proposition \ref{Prop:lem:c}: assume $\otcA^{[m]} \to \ocA$ is a  covering of schemes or stacks which is highly ramified over $\ocA \smallsetminus \cA$; it is shown in Proposition \ref{Prop:total-ram} that $\otcA_g^{[m]} \to \otcA_g$ is highly ramified. Assume $X \subset \tcA$ closed, with resolution $X'$,  compactification $\overline X'$ and covering $\oX'_m$. Then all ramifications of $\oX'_m  \to \overline X'$ are bounded from below. We summarize this principle as follows:

\begin{theorem}\label{Th:principle} Let $\pi_m:\ocA^{[m]} \to \ocA, m\in \NN$ be a collection of covers of smooth Deligne--Mumford stacks which are $m$-highly ramified over a normal crossings divisor $E \subset \ocA$ (see Definition \ref{Def:m-highly-ramified}). Let $X\subset \ocA \smallsetminus E$  be closed and of logarithmic general type. Then there is $m_0$, depending only on $X$, such that for all $m>m_0$, every irreducible component of $\pi_m^{-1} X$ is of logarithmic general type.
\end{theorem}




\subsection{Discussion: arithmetic} \label{Sec:arithmetic}


	\subsubsection*{Context: elliptic curves}

	As mentioned already, in 1969 {\sc Manin} proved uniform boundedness of $K$-rational $p^r$-torsion on elliptic curves ($K$ a number field, $p$ a fixed prime)~\cite{Manin}. In modern terms, {\sc Manin}'s result follows because the modular curve $X_1(p^r)$ has  genus $>1$ for appropriate $r$, so by {\sc Faltings}'s theorem there are only finitely many elliptic curves with a $K$-rational $p^r$-torsion point. For each of these finitely many curves, $r$ is bounded by  the {\sc Mordell-Weil} Theorem. Alternatively, choosing a prime $\ell\neq p$ of good reduction, the prime-to-$\ell$ torsion injects into the group of points of the reduced curve, and is hence finite. Hence the $p$-primary torsion of these curves is uniformly bounded. 
	Remarkably, {\sc Manin} proved his result prior to {\sc Faltings}'s proof of {\sc Mordell}'s conjecture, using {\sc Demjanenko}'s method to show that $X_1(p^r)(K)$ is finite. 

	Nearly a decade later, {\sc Mazur} \cite{Mazur} proved uniform boundedness for rational torsion across all elliptic curves over $\QQ$: if $P$ is an $m$-torsion $\QQ$-point on an elliptic curve over $\QQ$ then $m\leq 12, m\neq 11$. This deep work makes use of {\sc Manin}'s theorem.

	{\sc Kamienny} \cite{Kamienny} proved uniform boundedness for all $K$-rational torsion on elliptic curves over \emph{all} quadratic fields $K$, and {\sc Kamienny} and {\sc Mazur}  \cite{Kamienny-Mazur} proved uniform boundedness for all rational torsion on elliptic curves over fields of low degrees. In 1996, {\sc Merel} \cite{Merel} proved uniform boundedness for all torsion on elliptic curves over fields of any fixed degree.
	
	\begin{theorem}[{\sc Merel}] 
		Given an integer $d$ there is $m$ such that the torsion group of an elliptic curve over a number field of degree $\leq d$ has order $\leq m$.
	\end{theorem}


	\subsubsection*{Context: one-parameter families of abelian varieties} 
	
	Our results do not yield bounds for torsion points on abelian varieties, even conditionally, since we require full-level structure. {\sc Cadoret} and {\sc Tamagawa} \cite{Cadoret-Tamagawa-uniform} prove the exact analogue of {\sc Manin}'s result, and several variations thereof, for one-parameter families of abelian varieties.

	\begin{theorem}[{\cite[Theorem 1.1]{Cadoret-Tamagawa-uniform}}] 
		Given a prime $p$ and a one-parameter family $\cA \to S$ of abelian varieties of dimension $g$ over a number field $K$, there is an integer $r$ such that for all $s\in S(K)$ the order of $\cA_s[p^{\infty}](K)$ is  at most $p^r$.
	\end{theorem}

	{\sc Cadoret} and {\sc Tamagawa} use {\sc Faltings}'s theorem for appropriate towers of covers of $S$. The genus of these covers need not grow as one goes up a tower, e.g., if $\cA$ contains an isotrivial abelian subvariety. However~\cite[Theorem 1.2]{Cadoret-Tamagawa-uniform} shows this is the only condition that can stunt genus growth, and~\cite[Lemma~4.5]{Cadoret-Tamagawa-uniform} bounds the torsion in the isotrivial factors. 

	One can combine the various methods of {\sc Cadoret-Tamagawa,  Ellenberg-Hall-Kowalski, Hwang-To, Bakker-Tsimerman},  with the results here towards giving conditional bounds on something smaller than full-level structure of order $p^r$, perhaps even $p^r$-torsion points. We hope to address this in the near future. In \cite{Cadoret}, {\sc Cadoret} does precisely this for Hilbert modular surfaces.
	
	Closely related results are proved in \cite{EHK}: for instance,  \cite[Theorem 7]{EHK} shows that there are only finitely many  fibers with $p$ torsion for large $p$.


	\subsubsection*{{\sc Lang}'s conjecture}

	There are several conjectures surrounding rational points on varieties of general type often attributed to {\sc Lang}. The following is also known as the {\em {\sc Bombieri--Lang} conjecture}:

	\begin{conjecture}[{\sc Lang}'s conjecture]
		\label{conj:Lang}
		Let $X$ be a positive-dimensional variety of general type defined over a number field $K$. Then the set of rational points $X(K)$ is not Zariski-dense in $X$.
	\end{conjecture}

	This is a natural extension of {\sc Mordell}'s conjecture, as a curve is of general type precisely when its genus is $>1$. It is known at present for subvarieties of abelian varieties \cite{Faltings1,Faltings2} and few other cases. 
%
%
%
%
%


	\subsubsection*{From eventual hyperbolicity to the uniform power bound}

	The  following corollary of Theorem \ref{Th:eventual-hyperbolicity} provides a bridge to Theorem \ref{Th:main}. We again consider a subvariety $X \subset \cA_g$ defined over a number field $K$. 

	\begin{corollary}[of Theorem \ref{Th:eventual-hyperbolicity}]
		\label{Cor:not-dense}
		Let $X(K)_{[m]}$ be the set of $K$-rational points of $X$ corresponding to abelian varieties $A/K$ admitting full level-$m$ structure. 
		\begin{enumerate}
			\item Assume that {\sc Lang}'s conjecture~\ref{conj:Lang} holds and let $m>m_X$. If $\dim X \geq 2$ then $X(K)_{[m]}$ is not Zariski-dense in $X$. \label{dimgeq2}
			\medskip
			\item If $\dim X = 1$ and $m>m_X$ then $X(K)_{[m]}$ is (unconditionally) not Zariski-dense in $X$. \label{dim1}
			\medskip
			\item If $\dim X = 0$, then $X(K)_{[m]} = \emptyset$ for all $m \gg 0$. \label{dim0}
		\end{enumerate}
	\end{corollary}

	Corollary \ref{Cor:not-dense} is proven in \S\ref{Sec:reduction}. Theorem \ref{Th:main} follows by a simple argument involving noetherian induction, see \S\ref{Sec:proof}. 

While this paper was under review, {\sc Brunebarbe} posted the following result:

	\begin{theorem}[{\cite[Theorem 1.6]{Brunebarbe}}]
		\label{Th:strong-eventual-hyperbolicity} 
		Fix an integer $g\geq 1$. Then for $m>12g$ every subvariety of $\tcA_g^{[m]}$ is of general type.
	\end{theorem}	
This implies the following more uniform version of Theorem  \ref{Th:main}:
	\begin{corollary}[{\cite[Conjecture 1.10]{Brunebarbe}}]
		\label{Th:uniform-strong}
		Assume that {\sc Lang}'s conjecture~\ref{conj:Lang} holds. Fix an integer $g\geq 1$ and  $m>12g$. Let $K$ denote a number field. Then  there are only finitely many principally polarized abelian varieties $A/K$ of dimension $g$ with full level-$m$ structure.
	\end{corollary}

%


	\subsubsection*{Further variants} 

	There are variants of $\cA_g^{[m]}$ one can use, giving rise to other uniform power bounds. Let $G_K$ denote the absolute Galois group of a number field $K$. Fix a Galois representation $\rho$ of $G_K$ on $\ZZ_p^{2g}$. Such a representation reduces to representations $\rho_{p^r}$ of $G_K$ on the quotients $(\ZZ/p^r\ZZ)^{2g}$. Let $V(p^r)$ be the group scheme associated to $\rho_{p^r}$ (see~\cite[\S3~(7)]{Shatz}). A level-$\rho_{p^r}$ structure on an abelian variety $A$ of dimension $g$ is an isomorphism of $A[p^r] \cong V(p^r)$. An abelian variety has full level-$p^r$ structure if it has level $\rho_{p^r}$ for the representation $\rho = \ZZ_p^{g} \times\ZZ_p(1)^{g}$.

	We have the following result slightly generalizing  Theorem \ref{Th:main}:

	\begin{theorem}[Twisted uniform power bound]
		\label{Th:main-twisted}
		Assume that {\sc Lang}'s conjecture holds. Fix an integer $g$, a prime $p$, a number field $K$, and a Galois representation $\rho$ of $G_K$ on $\ZZ_p^{2g}$. Then there is an integer $r$ such that only finitely many principally polarized abelian varieties $A/K$ have level-$\rho_{p^r}$ structure.
	\end{theorem}

	If $\rho$ is the Tate module of an abelian variety, then the finite set implied by Theorem~\ref{Th:main-twisted} is not empty.


	\subsubsection*{Other towers} 
	
	The statement about full level-$p^r$ structure is made for convenience. Our methods require only having a sequence of levels $\{m_i\}$ with $m_i<m_{i+1}$ and $m_i\mid m_{i+1}$. Let us call such a sequence of integers a {\em level tower}.

	\begin{theorem}[Uniform tower bound]
		\label{Th:tower}
		Assume that {\sc Lang}'s conjecture holds. Fix an integer $g$, a level tower $\{m_i\}$, and a number field $K$. Then there is an integer $r$ such that no principally polarized abelian variety $A/K$ of dimension $g$ has full level-$m_r$ structure.
	\end{theorem}


\section{Proof of the uniform power bound and its variants}


\subsection{Proof of Corollary \ref{Cor:not-dense}} \label{Sec:reduction}

	Without loss of generality, we may assume that $X$ is irreducible. Let $Y = \overline{\pi_m^{-1}(X)(K)}$. Eventual hyperbolicity~\ref{Th:eventual-hyperbolicity} implies that each irreducible component of $\pi_m^{-1}(X)$ is of general type. Thus, in the case $\dim X \geq 2$, {\sc Lang}'s conjecture~\ref{conj:Lang} implies that $Y$ does not contain any irreducible component of $\pi_m^{-1}(X)$, so $\pi_m(Y)$ is a proper closed subset of $X$. On the other hand, $X(K)_{[m]} \subseteq \pi_m(Y)(K)$. This proves~\eqref{dimgeq2}. The same argument proves~\eqref{dim1} after replacing {\sc Lang}'s conjecture with Falting's theorem. If $X$ is a point, then the torsion subgroup of an abelian variety $A/K$ corresponding to $X$ is finite, as can be seen from  the {\sc Mordell-Weil} Theorem, or by reducing $A$ modulo two different primes of good reduction.
	This proves~\eqref{dim0}.\qed


\subsection{Proof of Theorem \ref{Th:main}} \label{Sec:proof}

	For a positive interger $i$, set 
\[
W_i = \pi_{p^i}\left(\overline{\cA_g^{[p^i]}(K)}\right).
\]
Then $W_1 \supseteq W_2 \supseteq \cdots$ is a descending chain of closed subsets of $\cA_g$. This chain must stabilize because $\cA_g$ is a Noetherian topological space. Say $W_n = W_{n+1} = \cdots$. 

	We claim that $W_n$ has dimension $\leq 0$. Suppose not, and let $X \subseteq W_n$ be an irreducible component of positive-dimension. Then Corollary~\ref{Cor:not-dense} \eqref{dimgeq2} and~\eqref{dim1} implies that $X(K)_{[p^m]}$ is not Zariski dense in $X$ for $m > \max\{\log_p m_X, n\}$. Thus, the set $\overline{X(K)_{[p^m]}} \subseteq W_m$ is a proper closed subset of $X$.Ê On the
other hand, $W_m = W_n$, so $X$ is also an irreducible component of $W_m$, and hence
$\overline{X(K)_{[p^m]}} = X$, a contradiction.

	Finally, if $W_n$ is a finite set of points, then apply Corollary~\ref{Cor:not-dense} \eqref{dim0} to $W_n$ to conclude that $W_n(K)_{[p^m]} = \emptyset$ for all $p^m \gg 0$.\qed

\subsection{Proof of Theorem~\ref{Th:main-twisted}} Follow the proof of Theorem \ref{Th:main}, mutatis mutandis. To wit, replace $\cA_g^{[p^i]}$ with $\cA_g^{[\rho_{p^i}]}$, the moduli space parametrizing abelian varieties with level-$\rho_{p^i}$ structure. We arrive at a finite number of abelian varieties because Corollary~\ref{Cor:not-dense}\,\eqref{dim0} does not hold for arbitrary $\rho$.


\subsection{Proof of Theorem~\ref{Th:tower}} Follow the proof of Theorem \ref{Th:main}, mutatis mutandis. Explicitly: replace $p^i$ with $m_i$ in the definition of $W_i$, as well as $p^m$ with $m_j$, where $m_j > \max\{m_X,m_n\}$.\qed


\section{Logarithmic hyperbolicity}

	\begin{proposition}
		\label{Prop:hyp-m}
		Fix $m\geq 3$ and let $X_m \subset \cA_g^{[m]}$ be a closed subvariety, $X_m'\to X$ a resolution of singularities, $X_m' \subset \overline X_m'$ a smooth compactification with $\overline X_m' \smallsetminus X_m' =: D_m $ a normal crossings divisor.  Then $K_{\overline X_m'}+D_m$ is big.
	\end{proposition}

	\begin{proof}
		This is a special case of~\cite[Theorem~0.1(ii)]{Zuo}. Quite generally, {\sc Zuo} shows that if $X$ is a smooth complex projective variety with a normal crossings divisor $D$, and if $X\smallsetminus D$ carries a polarized variation of Hodge structures whose corresponding period map is generically injective, then the pair $(X,D)$ is of logarithmic general type.  In our situation, $X'_m$ carries a polarized variation of Hodge structure, whose corresponding period map is generically injective, because the period map for $X_m$ is injective.
		
		\medskip
		
		We note that the result is also a special case of \cite[Theorem A(ii)]{Popa-Schnell}. 
	\end{proof}

	The moduli stack $\tcA_g$ has a morphism to its coarse moduli scheme $\cA_g$ \cite{Keel-Mori}. For any $m\geq 3$ the morphism $\cA_g^{[m]} \to \cA_g$ lifts to an {\em \'etale} morphism $\cA_g^{[m]} \to \tcA_g$: indeed given a morphism $S \to \tcA_g$ corresponding to an abelian scheme $A \to S$, one can identify $$\cA_g^{[m]}\times_{\cA_g} S = Isom_S(A[m], V/mV),$$  which is \'etale over $S$, since we are working in characteristic 0.

	\begin{proof}[Proof of Theorem \ref{Th:log-hyp}]
		Define $X_m := X \times_{\tcA_g} \cA_g^{[m]}$. This fibered product admits a resolution of singularities  $X'_m =  X' \times_{\tcA_g} \cA_g^{[m]}$, because $\cA_g^{[m]} \to \tcA_g$ is \'etale. Define  $\widetilde X_m'$ to be the normalization of $\overline X'$ in the total ring of functions of $X_m'$. These stacks fit into the diagram
	\[
		\xymatrix{
		X_m' \ar@{^(->}[r] \ar[d] & \widetilde X_m'\ar[d]^{\tilde\pi} \\ 
		X'  \ar@{^(->}[r] & \overline X'
		}.
	\]
Let $\widetilde D_m =   \widetilde X_m' \smallsetminus X_m'$.

		\begin{lemma}
			\label{lem:big}\ 
			\begin{enumerate}
				\item The embedding $X_m' \subset  \widetilde X_m'$ is a toroidal embedding. \label{item:toroidal}
				\medskip
				\item \label{item:divs} We have an equality of divisor classes 
					\[
						K_{\widetilde X_m'} + \widetilde D_m = \tilde\pi^*(K_{\overline X'}+D).
					\]
				\item The class $K_{\widetilde X_m'} + \widetilde D_m$ is big if and only if the class $K_{\overline X'}+D$ is big.\label{item:big}
			\end{enumerate}
		\end{lemma} 

		\begin{proof} 
			\eqref{item:toroidal} follows from {\sc Abhyankar}'s lemma. First, let $x\in \oX'$ be a point, and let $t_1\ldots,t_k$ be de defining equations of $D$ at $x$. Then the complete local ring at $x$ is of the form $\CC\llbracket t_1,\ldots,t_k, y_1,\ldots,y_l\rrbracket$, which is isomorphic to the completion of the coordinate ring of $\AA^k \times \GG_m^l$, a toric variety for the torus $\GG_n^k \times \GG_m^l$. Hence $X' \subset \oX'$ is a toroidal embedding. Let $X_m\in \tX'_m$ be a point over $x$. Since $X'_m \to X'$ is \'etale, {\sc Abhyankar}'s lemma \cite[Theorem X.3.6]{SGA1} says that there is an integer $n$ such that the completion $\hat\cO_{\tX'_m,x_m}$ is the ring of invariants of  the extension $\CC\llbracket \sqrt[n]{t_1},\ldots,\sqrt[n]{t_k}, y_1,\ldots,y_l\rrbracket$ by a subgroup $H$ of $\mu_n^k$ acting diagonally on the $t_i$. This is precisely the completion of $$\Spec \CC[ \sqrt[n]{t_1},\ldots,\sqrt[n]{t_k}, y_1,\ldots,y_l]/H,$$ an affine toric variety for the torus $ \GG_n^k \times \GG_m^l / H$, and therefore $X_m' \subset  \widetilde X_m'$ is a toroidal embedding as required. 
  
			For~\eqref{item:divs}, we continue with the notation above and note that the logarithmic differential $\frac{dt_1}{t_1}\wedge \cdots  \wedge\frac{dt_k}{t_k} \wedge dy_1 \wedge \cdots \wedge dy_l$ generates the stalk of logarithmic canonical differentials on the completion $\widehat{\oX'}_x$. Writing $u_i^m = t_i$, this differential pulls back to  $\frac{m\,du_1}{u_1}\wedge \cdots  \wedge\frac{m\,du_k}{u_k} \wedge dy_1 \wedge \cdots \wedge dy_l$, which generates the stalk on $\hat Y:=\Spf \CC\llbracket \sqrt[n]{t_1},\ldots,\sqrt[n]{t_k}, y_1,\ldots,y_l\rrbracket$. Since the pullback maps of stalks along $\hat Y \to \widehat{(\oX'_m)}_{x_m} \to \widehat{\oX'}_x$ are injective this element generates the stalk on the intermediate stage $\widehat{(\oX'_m)}_{x_m}$, as needed.
  
			Statement~\eqref{item:big} is a general fact about generically finite maps, see the argument of \cite[Theorem 1.5]{Debarre}: let $f\colon Y \to X$ be a dominant generically finite map  between reduced projective schemes of pure dimension $d$ and let $A$ be a divisor on $X$. Then $A$ is big if and only if $f^*A$ is big. As {\sc Robert Lazarsfeld} put it, this statement - in the form $\operatorname{vol}( f^* A) = \deg (f) \operatorname{vol}(A)$ - should have been in his books but isn't. To prove this it suffices to consider the case when $X$ is irreducible.   First, if $A$ is big then $h^0(X,\cO_X(mA)) = C m^d+ l.o.t$, and $$H^0(X,\cO_X(mA)) \subset H^0(Y',\cO_Y((f')^*(mA)))$$ for any irreducible component $Y'$ of $Y$, so  $$h^0(Y',\cO_{Y'}((f')^*(mA)))\geq  C m^d+ l.o.t$$ as needed. 

			In the other direction we may assume that $Y$ is also irreducible and write $\deg f=r$. If  $f^*A$ is big then  $$h^0(Y,\cO_{Y'}(f^*(mA)))=  C' m^d+ l.o.t.$$ But $$H^0(Y,\cO_{Y}(f^*(mA))) = H^0(X, f_*\cO_Y\otimes \cO_X(mA))$$ so $$h^0(X, f_*\cO_Y\otimes \cO_X(mA)) = C' m^d+ l.o.t.$$ Choosing an ample sheaf $\cO_X(1)$, there is an exact sequence  $$0 \to \cO_X(-n)^r \to f_*\cO_Y \to T\to 0$$ for some sheaf $T$ supported in dimension $\leq d-1$. Since 
			\[
				h^0(X, T \otimes \cO_X(mA)) \leq O(m^{d-1})
			\]
it follows that 
			\[
				h^0(X, \cO_X(-n)^r \otimes \cO_X(mA))\geq  C' m^d+ l.o.t,
			\]
hence
			\[
				h^0(X, \cO_X(mA))\geq h^0(X, \cO_X(-n)\otimes \cO_X(mA))\geq  \frac{C'}{r} m^d+ l.o.t
			\]
as needed.
		\end{proof}
		
		To complete the proof of Theorem \ref{Th:log-hyp},  let $\phi\colon \overline X_m'\to\widetilde X_m'$ be a toroidal resolution of singularities: it exists by \cite[Theorem 11*, p. 94]{KKMS}. Let $D_m = \overline X_m' \smallsetminus X_m'$, which is a normal crossings divisor. We have $\phi^*(K_{\widetilde X_m'} + \widetilde D_m) = K_{\overline X'_m}+  D_m$ \cite[p. 268]{Mumford}. We note that $\phi_* \cO_{\overline X_m'} = \cO_{\widetilde X_m'}$ since $\widetilde X_m'$ is normal. Therefore for every integer $n$ we have an equality 
		\[
			H^0( \overline X_m', \cO_{\overline X_m'}(n\cdot (K_{\overline X'_m}+  D_m))) = H^0( \widetilde X_m', \cO_{\widetilde X_m'}(n\cdot (K_{\widetilde X'_m}+  \widetilde D_m))).
		\]  
In particular $K_{\widetilde X_m'} + \widetilde D_m$ is big if and only if $K_{\overline X'_m}+  D_m$ is big. By Proposition \ref{Prop:hyp-m}, $K_{\overline X'_m}+  D_m$ is big. It follows that $K_{\widetilde X_m'} + \widetilde D_m$ is big, and by Lemma~\ref{lem:big}\eqref{item:big}, the class $K_{\overline X'}+D$ is big, as required.
	\end{proof}


\section{Ramification and eventual hyperbolicity}

	\begin{proposition}[{\cite[Proof of Theorem 3.1]{Nadel}}] 
		\label{Prop:total-ram}
		Let $x$  be a complex point  of $\otcA_g$ with  complete local ring $\hat O_x= \CC\llbracket t_1,\ldots t_k,y_1,\ldots y_n\rrbracket$, with $t_i$ the defining equations of boundary components and $y_i$ coordinates along the stratum of $x$. Let $\tilde x$ be a point of $\ocA_g^{[m]}$ above $x$. Then $\hat O_{\tilde x}= \CC\llbracket u_1,\ldots u_k,y_1,\ldots y_n\rrbracket$ with $u_i^m = t_i$.
	\end{proposition}
We provide a proof, similar to Nadel's.

	\begin{proof}
		We follow \cite{Mumford}, especially the notation on pages 254-255 and argument on pages 271-272. Let $\sD$ be the space of complex $g\times g$ symmetric matrices with positive definite imaginary part. If $F\subset  \overline \sD \smallsetminus \sD$ is a rational boundary component, then its closure contains a 0-dimensional boundary component (a 0-cusp). Thus  we may as well assume that $x$ corresponds to a 0-cusp, and moreover $x$ is itself a 0-dimensional stratum of the smooth toroidal compactification; in particular in the notation above $k=g(g+1)/2$ and $n=0$. Since all rational 0-cusps lie in the same $GL_g(\ZZ)$ orbit, we may as well assume that the cusp is the point $(i\infty) I_g$ in $\overline \sD$.

		In this situation in the notation of \cite[p. 254]{Mumford}, we have
		\begin{align*}
			N(F) &= \left\{ \left(\begin{matrix} A & B \\ 0 & D\end{matrix} \right)\in  Sp_{2g}(\RR)\right\} \\
			W(F) = U(F) &= \left\{ \left(\begin{matrix} I_g & B \\ 0 & I_g\end{matrix} \right) \bigg|  B = B^T\right\} 
		\end{align*} 
In particular the integer $l$ occuring there is 0.

		In this case, in the notation of \cite[p. 255]{Mumford}, $\sD_F = \sD(F)$ is the space $U(F)_\CC$ of all symmetric $g\times g$ complex matrices, and $U(F)$ acts by translation by the real symmetric matrices  $B$. We have $\Gamma = Sp_{2g}(\ZZ)$, and $U(F) \cap \Gamma$ is the additive group of symmetric $g\times g$ integer matrices $B$. Let $\Gamma(m) = \ker(Sp_{2g}(\ZZ) \to Sp_{2g}(\ZZ/m\ZZ))$, and note that $U(F) \cap \Gamma(m)$ is the additive group of symmetric $g\times g$ integer matrices divisible by $m$.

		Finally, in the notation of \cite[p. 272]{Mumford}, we have $$\sD(F) / U(F)\cap \Gamma = \CC^{g(g+1)/2} / \ZZ^{g(g+1)/2} = (\CC^*)^{g(g+1)/2}.$$ The chart of the smooth toroidal compactification of $\cA_g$ is  $$(\CC^*)^{g(g+1)/2} \subset \CC^{g(g+1)/2}.$$ 

		Similarly $\sD(F) / U(F)\cap \Gamma(m) = \CC^{g(g+1)/2} / (m\ZZ)^{g(g+1)/2} = (\CC^*)^{g(g+1)/2}$, with similar chart in the smooth toroidal compactification of $\cA_g^{[m]}$. The group $\Gamma/\Gamma(m)\simeq (\ZZ/m\ZZ)^{g(g+1)/2}$ acts by $m$-th roots of unity on each variable.  
		\[
			(a_{i,j})\cdot  (u_{i,j}) = (\zeta_m^{a_{i,j}}\, u_{i,j}).
		\]
It follows  the quotient map of smooth toroidal compactifications
		\[
			\overline{\sD(F) / U(F)\cap \Gamma(m)} \to \overline{\sD(F) / U(F)\cap \Gamma}
		\]
raises each variable to the  $m$-th power, as required. 
	\end{proof}

\begin{definition} \label{Def:m-highly-ramified}
Let $\pi_m\colon \ocA^{[m]} \to \ocA$ be a finite map between smooth  Deligne--Mumford stacks of finite type over $\CC$. Fix a normal crossings divisor  $E \subset \ocA$.  The map $\pi_m$ is said to be $m$-highly ramified over $E$ at a complex point $x \in \ocA$ if for every irreducible component of $E$ with local parameter $t$ at $x$ and every $\tilde x \in \ocA^{[m]}$ above $x$ we have $t = \nu u^m$ for some $u,\nu \in \hat O_{\tilde x}$, with $\nu$ a unit.
\end{definition}

\begin{proposition}\label{Prop:lem:c}
Let $\pi_m\colon  \ocA^{[m]} \to \ocA$ be a finite and $m$-highly ramified map between smooth Deligne--Mumford stacks of finite type over $\CC$, with respect to a normal crossings divisor  $E \subset \ocA$.  Let $X \subset \ocA \smallsetminus E$ be a closed subvariety, with $X' \to X$ a desingularization, and $\oX'$ a smooth compactification such that $D := \oX' \smallsetminus X'$ is a normal crossings divisor, and such that the rational map $\oX' \dasharrow \ocA$ extends to a morphism of pairs $(\oX',D) \to (\ocA,E)$.  Let $\oX'_m \to \oX\times_\ocA\ocA_m$ be a resolution of singularities with projection $\pi_m^X\colon \oX'_m \to \oX'$.  There is a constant $c > 0$, depending only on $\oX'$, such that
		\[
			\left(\pi_m^X\right)^*D \geq cm D_m,
		\]
where $D_m \subset \oX'_m$ is the complement of $X'_m$.
\end{proposition}


	\begin{proof}
		We give a lower bound for the multiplicity of $\left(\pi_m^X\right)^*D$ along each irreducible divisor in $\oX'_m$ lying above $D$. To avoid explicitly computing $\oX'_m$, we shall work with valuations. Since such a divisor gives a discrete valuation on the function field $\kappa(\oX'_m)$, and hence a discrete valuation on $\kappa(\oX')$, we begin by fixing a discrete valuation ring $R$ in $\kappa(\oX')$, and we let $R_m$ be the integral closure of $R$ in $\kappa(\oX'_m)$. Let $z$ and $w$ be respective uniformizing parameters for $R$ and $R_m$, so that $z = \mu w^r$ in $R_m$ for some unit $\mu$.

		Let $D_i$ be irreducible divisors such that $D = \sum_{i=1}^\ell D_i$. Let 
		\[
			f\colon (\overline{X}',D) \to (\ocA,E)
		\]
be the map of pairs arising from the set-up, where $E = \sum_{j = 1}^k E_j$, for some irreducible divisors $E_j \subset \ocA$. 

		Let $x \in \oX'$ be a complex point with image $f(x) \in E$, and let $x_m \in \oX'_m$ be a point with $p_1(x_m) = x$. Write $s_i$ for a local equation of $D_i$ near $x$, and $t_i$ for a defining equation of $E_i$ near $f(x)$.  Write $f_m\colon \oX'_m \to \ocA^{[m]}$ for the second projection. Our assumption that $\ocA^{[m]} \to \ocA$ is $m$-highly ramified implies that near $\tilde x := f_m(x_m)$ we have $t_i = \nu_i u_i^m$ for units $\nu_i$, and thus 
		\begin{equation}
			\label{eq:mth-power}
			t_1\cdots t_k = \nu(u_1\cdots u_k)^m
		\end{equation}
is an $m$-th power up to a unit $\nu$.
 
For $j = 1,\dots,k$ define collections of integers $\{a_{ij}\}_{i = 1}^\ell$ so that
		\begin{equation}
		\label{eq:aijs}
			f^*E_j = \sum_{i=1}^\ell a_{ij} D_i
		\end{equation}
Note that $a_{ij} \geq 0$, and that for a fixed $i$, at least one $a_{ij}$ is strictly positive. This way,
		\begin{equation}
		\label{eq:ais}
			f^*E = \sum_{i = 1}^\ell a_iD_i,\quad\textup{where }a_i := \sum_{j = 1}^k a_{ij},\textup{ and }a_i > 0.
		\end{equation}
By~\eqref{eq:aijs} we know that
		\[
			t_j = \nu' s_1^{a_{1j}}\cdots s_\ell^{a_{\ell j}}\quad\textup{for }j = 1,\dots,k.
		\]
and thus combining~\eqref{eq:mth-power} and~\eqref{eq:ais} we obtain
		\[
			(u_1\cdots u_k)^m = \nu'' s_1^{a_1}\cdots s_\ell^{a_\ell}.
		\]
		On the other hand, for $i = 1,\dots,\ell$ we have equalities in $R$ of the form
		\[
			s_i = \mu_i z^{b_i},
		\]
where $\mu_i$ is a unit. Hence there is a unit $\mu' \in R$ such that
		\[
			(u_1\cdots u_k)^m = \mu'\cdot z^{\sum a_ib_i}.
		\]
In turn, we get an equality in $R_m$ of the form
		\[
			(u_1\cdots u_k)^m = \mu''\cdot w^{\left(r\sum a_ib_i\right)}
		\]
for some unit $\mu'' \in R_m$. We infer that $m \mid r\cdot(\sum a_ib_i)$, and so
		\[
			r\geq \frac{m}{\gcd\left( m, \sum a_ib_i\right)}.
		\]
(Note that $\gcd\left( m, \sum a_ib_i\right) > 0$ since $a_i > 0$ for all $i$, and at least one $b_i$ is positive since we are looking at irreducible divisors lying over $D$, hence divisors whose image must be contained in at least one $D_i$.)

		Next, $D$ is locally defined by $s_1\cdots s_\ell$, so
		\[
			s_1\cdots s_\ell = \mu''' w^{r\sum b_i},
		\]
and the order of $\left(\pi_m^X\right)^*D$ on $R_m$ is thus
		\[
			r\cdot\sum b_i \ \geq\  \frac{m}{\gcd\left( m, \sum a_ib_i\right)}\cdot\sum b_i
		\]
Let 
		\[
			c = \frac{1}{\max_j \{a_j\}},
		\]
so that $a_i\cdot c \leq 1$ for $i = 1,\dots,\ell$. Then
		\[
			c\cdot \sum a_i b_i \leq \sum b_i,
		\]
allowing us to conclude that
		\begin{align*}
			r\cdot\sum b_i\ &\geq\ \frac{m}{\gcd\left( m, \sum a_ib_i\right)}\cdot c \cdot\sum a_ib_i \\
			&=\ c\cdot \lcm\left(m,\sum a_ib_i \right)\geq c\cdot m.
		\end{align*}
\end{proof}

	\begin{proof}[Proof of Theorems~\ref{Prop:GRH} and \ref{Th:principle}] By Proposition \ref{Prop:total-ram} we have $\otcA_g^{[m]} \to \otcA_g$ is $m$-highly ramified in Theorem ~\ref{Prop:GRH}. 
		The morphism $\pi_m^X\colon \oX'_m \to \oX'$ is a generically finite morphism of smooth varieties \'etale away from a normal crossings divisor $D_m$ on $\oX'_m$ and $D$ on $\oX'$. Hence, by~\cite[Theorem~11.5]{Iitaka}, we have
		\[
			\left(\pi_m^X\right)^*(K_{\oX'} + D) \leq K_{\oX'_m} + D_m.
		\]
It follows that
		\[
			\left(\pi_m^X\right)^*\left(K_{\oX'} + (1-\epsilon)D\right)
\leq K_{\oX'_m} + D_m - \epsilon \left(\pi_m^X\right)^*D
		\]
Let $c$ be the constant furnished by Proposition~\ref{Prop:lem:c}, and pick $m_X$ so that $\epsilon > 1/cm_X$. Then for each $m > m_X$ we obtain
		\[
			\epsilon \left(\pi_m^X\right)^*D > \frac{1}{cm}\left(\pi_m^X\right)^*D > D_m.
		\]
We conclude that for such $m$
		\[
			\left(\pi_m^X\right)^*(K_{\oX'} + (1 - \epsilon)D) < K_{\oX'_m} + D_m - D_m = K_{\oX'_m},
		\]
as desired.
	\end{proof}
	
	\begin{proof}[Proof of Theorem~\ref{Th:eventual-hyperbolicity}]
		The theorem follows easily from Corollary~\ref{Cor:epsilon} and Theorem~\ref{Prop:GRH}.
	\end{proof}

\vspace{.2in}

\hspace{2.6in}
\begin{minipage}{3in}
\emph{``And so it went on until the morning,}

\emph{when all the straw was spun, }

\emph{and all the reels were full of gold.''}
\end{minipage}

\vspace{.2in}

\rightline{Jacob and Wilhelm Grimm}
\rightline{Rumpelstiltskin, in \emph{Children's and Household Tales}}

\bibliographystyle{plain}             
\bibliography{levels} 

\end{document}